\documentclass[11pt,reqno]{amsart} \usepackage{amssymb}
\usepackage{amsthm,amsfonts}
\usepackage{xspace}
\usepackage{amsmath}
\usepackage{setspace}
\usepackage{amscd}
\usepackage{tikz}
\usepackage{setspace}
\usepackage{enumerate}
\usepackage{graphicx}
\makeatletter
\@namedef{subjclassname@2010}{%
  \textup{2010} Mathematics Subject Classification}
\makeatother
\newtheorem{theorem}{Theorem}[section]
\newtheorem{lemma}{Lemma}[section]
\newtheorem{proposition}{Proposition}[section]

\theoremstyle{remark}

\newtheorem{definition}{Definition}
\newtheorem{conj}{Conjecture}[section]

\numberwithin{equation}{section}

\newcommand{\N}{\mbox{$\mathbb{N}$}}

\newcommand{\Q}{\mbox{$\mathbb{Q}$}}
\newcommand{\F}{\mbox{$\mathbb{F}$}}

\begin{document}
\title{Iterates of polynomials over 
$\F_q(t)$ \\ and their Galois groups}
\author{Sushma Palimar}
\address{ 
Department of Mathematics,\\ 
Indian Institute of  Science,\\
Bangalore, Karnataka, India.}
\email{ sushmapalimar@gmail.com}
\subjclass[2020]{37P15; 11R32, 11T55,12F10,12F20,}
\keywords{Galois group; Odoni's conjecture; rational function field;  finite fields.}
\begin{abstract}
A conjecture of Odoni stated 
over  Hilbertian fields $K$ of characteristic zero asserts that
for every positive integer $d$, there exists a polynomial $f\in K[x]$ of degree $d$ such that for every positive integer $n$, each  iterate $f^{\circ n}$ of $f$ is irreducible 
and the Galois group of the
splitting field of $f^{\circ n}$
is isomorphic to $[S_d]^{n}$, 
the $n$ folded iterated  wreath product of the symmetric group $S_{d}$. 
We prove an analogue this conjecture over $\F_q(t)$, the field of rational functions in $t$ over a  finite field
$\F_q$ of odd cardinality. 
We present some  examples and see that most polynomials in $\F_q[t][x]$ satisfy these conditions.
\end{abstract}
\maketitle
\section{Introduction}
 Let  $K$ be a Hilbertian field and
$\phi(x)\in K[x]$ be  a polynomial of degree $d>1$. For a fixed $n\in \N, $ let $\phi^{\circ n}=\phi\circ\dots \circ \phi$, ($n$ times)
denote the $n$-th iterate of $\phi$.
Assume that  $\phi^{\circ n}$ is separable for all $n\geq 1.$
For $t\in K$,  consider the Galois group of polynomials of the form
 $\phi^{\circ n}(x)-t$ over $K$.
 Such groups are called arboreal Galois groups. Given such a polynomial, its arboreal Galois representation is surjective if and only if its splitting field is an 
$S_d$ extension.
 The study of the arboreal Galois groups of these field extension over fields of characteristic $0$ is started by R.W.K. Odoni \cite{Odo}. 
 An explicit source of arboreal
Galois representations is given by the Galois action on the roots of the
iterates of a given polynomial.
Let $K_n$ denote the splitting field of $\phi^{\circ n}(x)-t$ over $K$. 
The $d^n$ distinct roots of the $n$th iterate  of $\phi$ can be identified with $d^{n}$ vertices
at level $n$ of the $d$-ary rooted tree $T$ in such a way that Galois group of $\phi^{\circ n}$ embeds in $\mathrm{Aut}(T_n)$, where $\mathrm{Aut}(T_n)$ is the group of automorphism of the $d$-ary rooted  tree of height $n$. 
A standard result in group theory is that, $\mathrm{Aut}(T_n)\cong [S_d]^{n}$, the $n$-folded iterated
wreath product of the symmetric group 
$S_{d}$. For a fixed $n\in \N,$  
the Galois
extension $K_n/K$ is maximal when $\mathrm{Gal}(K_n/K) = \mathrm{Aut}(T_n)$.
We call the relative Galois extension $K_{n}/K_{n-1}$ is maximal when 
$\mathrm{Gal}(K_{n}/K_{n-1})=S_d^{d^{n-1}}$.
Then it follows by induction and comparison of degrees that $\mathrm{Gal}(K_n/K)\cong [S_d]^{n}$ 
for all $n\in \N$, (Jones, \cite{rj}).
A Hilbertian field is a field $K$ which satisfies Hilbert irreducibility Theorem (HIT).
Examples of Hilbertian fields include $\Q$,
number fields, and finite extensions of $K(t)$ for any field $K$. In general for any field $K$ and an indeterminate  $u$ 
  it is known that
$K(u)$ is Hilbertian.
\begin{theorem}\label{HIT}
(HIT) Let $k$ be a Hilbertian field and $u_1,\dots, u_r, x$ be independent indeterminates over $k$. Suppose that 
$f(u_1,\dots, u_r, x)\in k[u_1,\dots, u_r, x]$ be any monic, irreducible polynomial  of degree $n>0$ in $x$ and separable over $k(u_1,\dots, u_r)$. Then there is a Hilbert subset $H_t=(U_1,\dots,U_r)\subset k^{r}$ such that 
$\mathrm{Gal}(f(u_1,\dots, u_r, x), k(u_1,\dots, u_r ))\cong 
\mathrm{Gal}(f(U_1,\dots, U_r, x), k)\cong S_n.$
\end{theorem} 
Let $K$ be a field, for a separable polynomial $f \in K[X]$
of degree $d>0$ we define its Galois group 
$\mathrm{Gal}(f,K)$ to be the Galois group of its splitting
field over $K$. We will always interpret the Galois group $\mathrm{Gal}(f,K)$ as a subgroup of the
symmetric group $S_d$ via its action on the roots of $f$. The embedding into $S_d$ is well defined
up to conjugation.
\begin{definition}Wreath product of groups:\cite{jl}. Let $G$ and $H$ be $2$ permutation  groups and $r$ be the degree of $G$. The wreath product of $G$ by $H$, denoted by $G[H]$ or $H\wr_{r} G=H^{r}\rtimes G$ is the semidirect product of the Cartesian product of $r$ copies of $H$ by $G$.
The $n$th wreath power of a group $G$ is defined recursively as $[G]^{1}=G$ and $[G]^{n}=[G]^{n-1}[G]$.
\end{definition}
\begin{definition}
 Suppose $k$ is any field, for any $n\in \N$, let $f^{\circ n}(x)$ be the 
$n$-th iterate of $f(x)\in k[x]$. 
 An element $\alpha\in k$ is said to be wandering if,  $f^{\circ l}(\alpha)\neq f^{\circ m}(\alpha)$ for all $l>m>0$.
 We say that $\mathfrak{p}$ is a primitive prime divisor of $f^{\circ n}(\alpha)$ if 
 $v_{\mathfrak{p}}(f^{\circ n}(\alpha))>0$ and $v_{\mathfrak{p}}(f^{\circ m}(\alpha))\leq 0$ for all $m<n.$
\end{definition}
\begin{conj}\label{conj-odo} [Conjecture 7.5, Odoni, \cite{Odo}]
 For any Hilbertian field $\F$ of characteristic $0$, there exists a monic polynomial $f(x)\in \F[x]$
 of degree $d\ge 2$ such that the Galois group of the $n$-th iterate of 
 $f(t,x)$ over $\F$ is,
 \begin{equation}\label{co-odo}
\mathrm{Gal}(f^{\circ n}(x),\F)\cong [S_d]^{n}, n\geq 1
\end{equation}
\end{conj}
Odoni proved his conjecture 
 for  $\phi(x)=x^{2}-x+1$  over a field $K$ of characteristic $0$ and showed that this polynomial has surjective arboreal representation (\cite{Odo}).  
 Conjecture \ref{conj-odo}  for generic polynomials  over an algebraically closed field $\F$ of characteristic $p>0$ was proved by Juul \cite{jl}. 
 While, Looper, \cite{nc} established this result over the field $\F=\Q$, with all prime degrees $d$.
Looper's result was independently extended in [\cite{bk},\cite{js1}].
 Kadets \cite{bk}, 
showed that Odoni's conjecture holds over any number field. 
Specter \cite{js1} proved Odoni's conjecture for all number fields,  more generally, for all algebraic extensions $K/\Q$ that are
unramified outside of a finite set of primes.
\section{Galois group of polynomials over  ${\F}_q(t).$} 
Odoni's conjecture \ref{conj-odo}, asserts that the Galois group of the field extension formed
by adjoining the roots of the $n$-th iterate is as large as possible.
We prove an analogue of this Conjecture over $\F_q(t)$. 
Suppose
$f(t,x)$ in $\F_q[t][x]$ 
is 
an irreducible, separable polynomial over $\F_q(t)$, which is  monic in $x$ of degree $d>2$   such that $\mathrm{char}\F_q=p\neq 2$, $p\nmid d(d-1)$. 
 Let $L$ be the splitting field of $f(t,x)$ over  $\F_q(t)$ and, 
 $k=\overline{\F}_q$  be an algebraic closure of $\F_q$ in some algebraic closure $ \Omega\supset L$ of $\F_q(t)$. 
Suppose, 
$\F$ is the splitting field of $f(t,x)$ over $k(t)$ in $\Omega$.
    We know that,
\begin{equation}\label{geom}
 \mathrm{Gal}(f(t,x),k(t)) =\mathrm{Gal}(L\cdot k/k(t))\leq \mathrm{Gal}(f(t,x),\F_q(t)) \end{equation}
 i.e., by passing to an algebraic closure of $\F_q$, Galois groups can
only decrease in size. 
So we may replace $\F_q$ with an algebraic closure $k$ of $\F_q$ of characteristic $p>2$ and prove
the result in this case. 
\subsection{Results}
\begin{theorem}\label{[S_n]}
Let $f(t,x)\in k[t][x]$ be a monic, irreducible  polynomial of degree $d>2$ in $x$, separable over $k(t)$ such that $p\neq 2$ and $p\nmid d(d-1)$. Then the Galois group $\mathrm{Gal}(f(t,x),k(t))\cong S_d,$ the symmetric group $S_d.$
  \end{theorem}
   For a fixed $n\in \N$, we denote the $n$-th iterate of $f(t,x)$ as, 
 \begin{displaymath}
f^{\circ n} (t,x):=f^{\circ n-1} (f(t,x)),\quad n\geq 1. 
\end{displaymath}
 and the splitting field of $f^{\circ n}(t,x)$ is denoted as $K_n$.
\begin{theorem}\label{S_n:ff}
 Let $f(t,x)\in k[t][x]$ be a monic, irreducible polynomial of degree $d>2$ in $x$, separable over $k(t)$, such that $p\neq 2$, $p\nmid d(d-1)$ and the group $\mathrm{Gal}(f(t,x),k(t))$ is isomorphic to $ S_d$.  Suppose that $f(t,x)$ has some critical point $a$ of multiplicity one in the algebraic closure $\Omega$ of
$k(t)$ such that
 $f^{\circ l}(t,a)\neq f^{\circ m}(t,b)$ for all $1\leq l\leq m\leq n\in \N$ unless $l=m$ and $a=b$. 
 Under these conditions for all $n\geq 1$,
 $\mathrm{Gal}(f^{\circ n}(t,x),k(t))\cong [S_d]^{n}$.
 \end{theorem}
 We may note that, Theorems \ref{[S_n]} and  \ref{S_n:ff} hold over any algebraically closed field $k$ of characteristic $p\geq 0$ and not just over $\overline{\F}_q$.
Proof of Theorem \ref{S_n:ff} proceeds by inductive hypothesis, by showing that for a fixed $n\in \N$, the  Galois group of $K_n$  over $K_{n-1}$ is $S_{d}^{d^{n-1}}$ for all $n \geq 1.$
In \S \ref{mo_nmo},
we  consider the examples of  both Morse and non-Morse polynomials
and see that most polynomials in $\F_q[t][x]$ satisfy Odoni's conjecture.
\subsection{Proof of Theorem \ref{[S_n]}}
 By hypothesis,
$f(t,x)\in k[t][x]$ is   a monic, irreducible polynomial in $x$ of degree $d>2$, separable over $k(t)$, $\mathrm{char}(k)=p\neq 2$ and
$p\nmid d(d-1)$.  Then $\F/k(t)$  is a Galois extension and the group $\mathrm{Gal}(f(t,x),k(t))$ is a transitive subgroup of $S_d$.
\subsection{Places of  Galois extension $\F/k(t)$}\label{places}
The ramified primes of the Galois extension $\F/k(t)$ are the pole $P_{\infty}$ of $t$ and the places corresponding to the zeros of the derivative of $f(t,x)$.
The primes of  $\F_q(t)$ are given by monic, irreducible polynomials in $\F_q[t]$ and the place above $t=\infty$. 
It is known that, for $k=\overline{\F}_q$ with one exception, the primes 
 of $k(t)$ are in one-to-one correspondence with linear polynomial $P_c=t-c$ in $k[t]$ for some $c\in k$. The exceptional point is the unique pole $P_{\infty}$ of $t$, (the infinite prime of $k(t)$) corresponding to $(1/t)$-adic valuation on $k[1/t]$.
 It is useful to note
 the following lemma by Hayes.
 \begin{lemma}(Hayes)\label{hayes}
 Let $K$ be the splitting field of $x^{d}+x-t$ over $k=\F_p(t)$. Then the group $\mathrm{Gal}(x^{d}+x-t,k)$ is isomorphic to  $S_d$, the symmetric group on $d$ letters under the following conditions.
 \begin{enumerate}
 \item The prime $P_{\infty}$ ramifies in $K.\overline{\F}_p/k.\overline{\F}_p$; if $\mathrm{char}\F_p\nmid d$ then the ramification index $e(Q_\infty|P_\infty)=d$ for any prime $Q_{\infty}$ of $ K.\overline{\F}_p $ dividing $P_{\infty}$.
  \item  If $p\nmid d(d-1)$ and if $P_c\neq P_{\infty}$
 ramifies in $K.\overline{\F}_p/k.\overline{\F}_p$, then the decomposition group $D_{\mathfrak{p}}$ of any prime $\mathfrak{p}$ of $K.\overline{\F}_p$ dividing $P_c$ is cyclic of order 2. Further if $\sigma \in \mathrm{Gal}(K.\overline{\F}_p/k.\overline{\F}_p)$ generates $D_{\mathfrak{p}}$ then the permutation induced by $\sigma$ on the roots of $x^{d}+x-t$ is a transposition.
 \end{enumerate}
\end{lemma} 
 Although Hayes proves this explicitly only to 
 find the Galois group of $x^{d}+x-t$ over $\F_p(t)$ to provide an alternate proof for the theorem of Birch and Swinnerton-Dyer, these results  can be  adapted to the general case.
  Let us find  the ramified primes in the Galois extension $\F/k(t)$.
Denote by,
\begin{displaymath}
 \Delta_{f}(t)= \mathrm{disc}_x(f(t,x)) 
 \end{displaymath} the 
 discriminant of $f(t,x)$, which is a nonzero  function of $t$.
The  primes of $k(t)$ that
ramify in the splitting field $\F$ of $f(t,x)$  are the divisors of the discriminant $\Delta_{f}(t)$ and  the infinite prime $P_{\infty}$ of $k(t)$.
The ramification index of any prime $Q_{\infty}$ of $\F$ dividing $P_{\infty}$ is $d$, i.e., 
$e(Q_{\infty}|P_{\infty})=d$. The pole $P_\infty$  has a tame  ramification and the inertia group is generated by an
$d$-cycle, and $Q_{\infty}$ is the only prime
lying above $P_{\infty}$.
Separability of $f(t,x)$ implies,
$f(t,x)$ has no multiple (double) roots, in any extension of $k(t)$, so that,
any  prime $P_c\neq P_{\infty}$ of $k(t)$,  dividing the discriminant has ramification index less than 3.
 Therefore if $\F/k(t)$ ramifies over $P_c$,  by Hensel's Lemma, $f(t,x)$ splits over the completion $K_{c}$ of $k(t)$ at $P_c=t-c$ into a quadratic factor and $(n-2)$ linear factors. Thus
the inertia group is
 generated by
transpositions, hence not contained in $A_d$, the alternating group of order $d$. 
Therefore, the group $\mathrm{Gal}(f(t,x),k(t))$ is a transitive subgroup of $S_d$ generated by transpositions, hence $\mathrm{G}=S_d$ by [Lemma 4.4.4, \cite{JP}].  Alternatively, the double transitivity property of Galois group 
can be invoked 
to prove that $\mathrm{G}=S_d$, which completes Theorem \ref{[S_n]}.
\subsection{Doubly transitive Galois groups} \label{doubly}The double transitivity of the group
$\mathrm{Gal}(f(t,x),k(t))$ 
is established by virtue of the method of ``throwing away roots'' 
by Abhyankar [\cite{ssa1} \S 4].
Let  $\alpha_1,\dots,\alpha_d$ be the $d$  distinct roots of $f(t,x)=0$ over $\F$; then,
$\F=k(t,\alpha_1,\dots,\alpha_d)$ and 
$f(t,x)$ splits into $d$ distinct linear factors $(x-\alpha_1)\dots(x-\alpha_d)$ in $\F[x]$ such that,  any two linear factors 
$(x-\alpha_i)$ and $(x-\alpha_j)$ for $i\neq j$ are pairwise  coprime in $\F[x]$. 
Eliminating a root of $f(t,x)=0$, say $x=\alpha_1$ 
gives,
\begin{equation}
 f_1(t,x)=\frac{f(t,x)}{(x-\alpha_1)}\in k(\alpha_1)[t,x]
\end{equation}
  $f_1(t,x)$ is a separable  polynomial over $k(\alpha_1)(t)$ of degree $d-1$ in $x$ and has no zero in $ k(\alpha_1)[t][x]$, therefore irreducible  in $k(\alpha_1)[t][x]$.  
Consequently $f(t,x)$ and $f_1(t,x)$ are irreducible in $k[t][x]$ and $k(\alpha_1)[t][x]$ respectively.
Thus, the stabilizer of the root 
$x=\alpha_1$ in the Galois group of $f_1(t,x)$ acts transitively on other roots,
 implying the Galois group  of $f(t,x)$ over $k(t)$ is doubly transitive. Since any doubly transitive group action is primitive, the Galois group $\mathrm{Gal}(f(t,x),k(t))$ is primitive. 
By the fact that, any primitive permutation group containing a transposition is symmetric, we have
\begin{equation}\mathrm{Gal}(f(t,x),k(t))\cong S_d
\end{equation}
This completes Theorem \ref{[S_n]}.
\section{Galois groups of iterates of Polynomials}\label{prf_th_[sn]}
Let us recall from Theorem  \ref{[S_n]} that $f(t,x)\in k[t][x]$ is a monic, irreducible, separable polynomial in  $x$ of degree $d>2$,  $k=\overline{\F}_q$, an algebraic closure  of ${\F}_q$, $\mathrm{char}(k)=p\neq 2$, $p\nmid d(d-1)$  and $\mathrm{Gal}(f(t,x),k(t))\cong S_d$.
By our notation, for a fixed $n\in \N$, the $n$-th iterate of $f(t,x)$ is denoted as $f^{\circ n}(t,x)$, defines as:
\begin{displaymath}
f^{\circ n}:=f^{\circ n-1}(f(t,x))
\end{displaymath}
and  the splitting field of $f^{\circ n}(t,x)$ over $k(t)$ is denoted as $K_n$.
Here, $f^{\circ 1}(t,x)$ is considered as $f(t,x)$ and $K_0=k(t).$
From the hypothesis of the Theorem \ref{S_n:ff}, $f(t,x)$ has some critical point
${a}$ of multiplicity one in some algebraic closure $\Omega$ of $k(t)$,   
  such that $f^{\circ l}(t,a)\neq f^{\circ m}(t,b)$ for all $l\leq m\leq n$ unless $m=l$ and $a=b$.  
  Aim of this section is to show that $K_{n}/K_{n-1}$ is maximal for all $n\geq 1.$
The reasoning relies on a  basic
fact of algebra known as,  Capelli's Lemma, which we will use many times
throughout the paper. We state it below without proof.
\begin{lemma}(Capelli's Lemma)\label{capl}
 Let $K$ be any field and let $f, g \in K[x]$.
Suppose $\alpha \in \overline {K}$ is any root of $f$. Then $f(g(x))$ is irreducible over $K$ if and
only if both $f(x)$ is irreducible over $K$ and $g(x)-\alpha$ is irreducible over $K(\alpha)$.
\end{lemma}
We recall that the discriminant of a monic separable polynomial $F(t)$ of degree $n$ is defined by the resultant of $F$ and the derivative $F^{\prime}$ of $F$.
 If $\tau_1,\dots,\tau_n$ are the distinct roots of $F(t)=0$, the discriminant is
\begin{equation}\label{checker2}
 \mathrm{disc}(F)= \mathrm{Res}(F,F^{\prime})=(-1)^{{n(n-1)}/2}\prod_{j=1}^{n}F^{\prime}(\tau_j)
\end{equation}
\begin{proposition}\label{lin_sep}
 With the above notation, for any positive integer $n$,  $f^{\circ n}(t,x)\in k[t][x]$ is  a monic irreducible polynomial  of degree $d^{n}$ in $x$ which is  separable over $k(t)$. Moreover,  the Galois extension $K_n/K_{n-1}$  is the compositum of  $d^{n-1}$ linearly disjoint field extensions $M_{1}^{\prime}/K_{n-1},\dots ,M_{d^{n-1}}^{\prime}/K_{n-1}.$
\end{proposition}
\begin{proof}
We prove the statement using induction on $n$.
The result holds true for $n=1$ by hypothesis. We use Capelli's Lemma for $n\geq 2.$
Suppose that the result holds true for $  n-1$. Let $a_1,\dots,a_{d^{n-1}}$ be the $d^{n-1}$ distinct roots of $f^{\circ n-1}(t,x)=0$ in the splitting field $K_{n-1}$. 
Then,
\begin{equation}\label{kn+1prod}
 f^{\circ n}(t,x)=f^{\circ n-1}(f(t,x))=\prod_{i=1}^{d^{n-1}}f(t,x)-a_i
 \end{equation}
 The $d^{n-1}$ factors $f(t,x)-a_1,\dots,f(t,x)-a_{d^{n-1}}$  on the right  of (\ref{kn+1prod}) 
are distinct monic polynomials in $x$ of degree $d>2$ which are coprime  in $K_{n-1}[x]$. 
 Since $a_i$ is transcendental over $k$,  $f(t,x)-a_i$ is irreducible over $k(a_i)=k(t,a_i)$ for $1\leq i\leq d^{n-1}$.
Hypothesis on $f(t,x)$ ensures, the derivative $f^{\prime}(t,x)$ is not identically zero, which implies each factor $f(t,x)-a_i$ is
separable over $k(t,a_i)$, which has 
no root algebraic over $k(t)$ while, each root of $f^{\prime}(t,x)$ is algebraic over 
$k(t)$. Thus, $f^{\circ n}(t,x)$ is irreducible  and separable over $k(t)$.
To prove the second part, let $M_i$ be the splitting field of 
 $f(t,x)-a_i$ over $k(t,a_i)$  for each $i=1,\dots,d^{n-1}$ respectively.
 The $d^{n-1}$ distinct roots $a_1,\dots,a_{d^{n-1}}$   of $f^{\circ n-1}(t,x)$ are  transcendental over 
$k$, so that,
\begin{equation}\label{mors-eq}
\begin{split}
  & \mathrm{Gal}(f(t,x),k(t)) \cong  \mathrm{Gal}(f(t,x)-a_i,k(t,a_i))\cong S_d
  \text{ for } 1\leq i\leq d^{n-1}.\\
  &\text{i.e.,}\\
 & \mathrm{Gal}(K_1/k(t)) \cong \mathrm{Gal}(M_i/k(t,a_i))\cong S_d \text{ for } 1\leq i\leq d^{n-1}.\\
  \end{split}
\end{equation}
 The polynomials $f(t,x)-a_1,\dots,f(t,x)-a_{d^{n-1}}$  are distinct and relatively prime in $K_{n-1}[x]$. Therefore,  their splitting fields  
  \begin{equation}\label{disj_spli}
  M_1/k(t,a_1), \dots, M_{d^{n-1}}/k(t,a_{d^{n-1}})
 \end{equation}                                                                             
have distinct ramifications hence, $ M_1/k(t,a_1), \dots, M_{d^{n-1}}/k(t,a_{d^{n-1}})$ are linearly disjoint over $k(t)$.
 Changing the base by $K_{n-1},$ let us denote,
\begin{equation}
 M_i^{\prime}:=M_iK_{n-1} \text{ for } 1\leq i\leq d^{n-1}.
\end{equation}
The extension $M_iK_{n-1}/K_{n-1}$ is Galois and $M_{1}^{\prime},\dots,M_{d^{n-1}}^{\prime}$ are linearly disjoint over $K_{n-1}$.
Thus the splitting field $K_{n}$ of $f^{\circ n}(t,x)$ over $K_{n-1}$ is the  compositum of $d^{n-1}$ linearly
disjoint splitting fields over $K_{n-1}$ given by, 
\begin{equation}\label{lin-hol}
K_{n}/K_{n-1}= M_{1}^{\prime}/K_{n-1}\cdots M_{d^{n-1}}^{\prime}/K_{n-1} \text{ for all } n\geq1.
\end{equation}
  \end{proof}
  \begin{proposition}\label{pal1} 
The  group 
\begin{displaymath}
\mathrm{Gal}(M_iK_{n-1}/K_{n-1})= \mathrm{Gal}(M_i^{\prime}/K_{n-1})\cong S_d \text{ for }1\leq i \leq d^{n-1}.\end{displaymath}
\end{proposition} 
\begin{proof}
From the above discussion,
\begin{displaymath}
\mathrm{Gal}(f(t,x)-a_i,k(t,a_i))\cong
\mathrm{Gal}(M_i/k(t,a_i))\cong S_d, \text{ for }1\leq i\leq d^{n-1}.
\end{displaymath} we find the group 
\begin{displaymath}\mathrm{Gal}(M_iK_{n-1}/K_{n-1})=\mathrm{Gal}(M_i^{\prime}/K_{n-1})\end{displaymath} is a normal subgroup of 
$
\mathrm{Gal}(M_i/k(t,a_i)).
$ But,
 \begin{displaymath}\label{la1-eq}
  \mathrm{Gal}(M_iK_{n-1}/K_{n-1})\cong \mathrm{Gal}(M_i/M_i\cap K_{n-1})=\mathrm{Gal}(M_i /k(t)(a_i))
  \end{displaymath}
  so that,
  \begin{displaymath}
  \mathrm{Gal}(M_iK_{n-1}/K_{n-1})
= \mathrm{Gal}(M_i^{\prime}/K_{n-1})= S_d \text{ for } 1\leq i\leq d^{n-1}.
\end{displaymath}
\end{proof}
\subsection{Squarefree Discriminants}
Let us denote,
$D_i:=\mathrm{disc}_{x}(f(t,x)-a_i)$, the discriminant of $f(t,x)-a_i$ for $i=1,2,\dots,d^{n-1}$. 
By Proposition \ref{lin_sep},  $D_{1},\dots,D_{d^{n-1}}$   are relatively prime in $K_{n-1}$ and,
\begin{equation}\label{ful_si}\mathrm{Gal}(f(t,x)-a_i,k(t,a_i))=\mathrm{Gal}(M_i/k(t,a_i))\cong S_d, \text{ for } 1\leq i\leq d^{n-1}.
\end{equation} The full symmetric group in (\ref{ful_si}) implies 
the  discriminant $D_i=\mathrm{disc}_x(f(t,x)-a_i)$ is squarefree for $i=1,2,\dots,d^{n-1}$. Therefore the
 product,    $D_1\cdots D_{d^{n-1}}$ is squarefree.
 \paragraph{}By Proposition \ref{pal1}, $\mathrm{Gal}(M_i^{\prime}/K_{n-1})=S_d$ for each $i=1,2,\dots, d^{n-1}$. Denote by
$E_i^{\prime}$, 
the unique quadratic extension  of $K_{n-1}$ inside  $M_i^{\prime}$,  given by $E_i^{\prime}=K_{n-1}(\sqrt{D_i})$.  Since the discriminants $D_1,\dots,D_{d^{n-1}}$ are relatively prime and squarefree, $E_{1}^{\prime},\dots, E_{d^{n-1}}^{\prime}$
are linearly disjoint over $K_{n-1}$. 
($E_i^{\prime}$ is
the fixed field of $A_d$ in $M_i^{\prime}$, where, $A_d$ is the alternating group on $d$ letters, which is a normal subgroup of $S_d$.)
\begin{proposition}\label{final-prop-lin}
Suppose $n$ is a positive integer. Then, any prime $P$ of 
$k(t)$ ramifying in $K_{n}$, does not ramify in $K_{n-1}$. 
 \end{proposition}
\begin{proof}
We prove the claim by induction. 
Let, $\alpha_1,\dots,\alpha_d$ be  the $d$ distinct roots of $f(t,x)=0$ as befoe.
The case $n=2$, follows  by noting that  $\mathrm{disc}_x(f(t,x))$ does not have any common factor with the   $\mathrm{disc}_x(f(t,x)-\alpha_i)$ for each $i=1,2,\dots,d$. 
Suppose the result holds true holds for $n-1.$
As before,  we write, 
\begin{equation}\label{x_0}
 f^{\circ n}(t,x)=f^{\circ n-1}(f(t,x))=\prod_{i=1}^{d^{n-1}}f(t,x)-a_i
\end{equation}
where $a_1,\dots,a_{d^{n-1}}$ are the $d^{n-1}$ distinct roots of $f^{\circ n-1}(t,x)=0$.
By  Proposition \ref{lin_sep},
each of $d^{n-1}$ factors $f(t,x)-a_i$ on the right of (\ref{x_0}) are
 distinct, irreducible, separable polynomials over $k(t)(a_i)$, monic in $x$ of degree $d>2$, which are relatively prime in $K_{n-1}[x]$. 
The splitting field $K_n$ of $f^{\circ n}(t,x)$ over $K_{n-1}$ is the compositum of 
 linearly disjoint Galois extensions $ M_{1}^{\prime}/K_{n-1},\dots, M_{d^{n-1}}^{\prime}/K_{n-1}$.
By inductive hypothesis the 
 $(n-1)$-th iterate of $f(t,x)$ is,
\begin{equation}\label{x_1}
 f^{\circ n-1}(t,x)=\prod_{j=1}^{d^{n-2}}f(t,x)-b_j
\end{equation}
where, $b_1,\dots,b_{d^{n-2}}$ are the $d^{n-2}$ distinct roots of 
$f^{\circ n-2}(t,x)$, such that each  of the $d^{n-2}$ factors $f(t,x)-b_j$ on the right of  (\ref{x_1}) are distinct, irreducible, separable polynomials over $k(t)(b_j)$, monic in $x$ of degree $d>2$, which are relatively prime in $K_{n-2}[x]$.
The splitting field $K_{n-1}$ of $f^{\circ n-1}(t,x)$ is the compositum of $d^{n-2}$ linearly disjoint Galois extensions $L_1^{\prime},\dots,L_{d^{n-2}}^{\prime}$  over $K_{n-2}$.
 \begin{equation}
 K_{n-1}=L_1^{\prime},\dots,L_{d^{n-2}}^{\prime}
 \end{equation} 
 Thus,  for each $i=1,2,\dots,d^{n-1}$, any prime divisor of the discriminant
  $\mathrm{disc}_x(f(t,x)-a_i)$ does not 
 divide the discriminant $\mathrm{disc}_x(f^{\circ n-1}(t,x))$. Therefore any prime  ramifying in $K_n$ 
  does not ramify in $K_{n-1}$.
  For if a prime $P$ of $k[t]$, ramifying in $K_{n}$, also ramifies in $K_{n-1}$, then it ramifies in one of the  $L_i^{\prime}$ for $1\leq i \leq d^{n-2}$, which is a contradiction for 
 \(M_1^{\prime},\dots,M_{d^{n-1}}^{\prime}\)  being linearly disjoint over $K_{n-1}.$
  Thus any prime $P$ of $k[t]$ ramifying in $K_n$
 does not ramify in any subextension of $K_n$ for all $n\geq 1$.
\end{proof}
 \begin{proposition}\label{rel-max}
The relative Galois extension $K_{n}/K_{n-1}$ is maximal, that is $\mathrm{Gal}(K_{n}/K_{n-1})\cong 
 S_{d}^{d^{n-1}}$ for all $n\geq 1.$ 
 \end{proposition}
 \begin{proof}
 We use induction on $n$. The  result holds true for $n=1$, i.e.,.  
\begin{displaymath}
 \mathrm{Gal}(f(t,x),k(t))=\mathrm{Gal}(K_1/K_0)\cong S_d.
\end{displaymath}
Suppose that the result holds true for $n-1$, i.e.,
\begin{displaymath}
 \mathrm{Gal}(f^{\circ n-1}(t,x), k(t))\cong [S_d]^{n-1}. 
\end{displaymath}
For any fixed $n\in \N$, the splitting field $K_n$ of $f^{\circ n}(t,x)$ over $K_{n-1}$ is the compositum of $d^{n-1}$ linealy disjoint Galois extensions given by,
 \begin{equation}\label{bvis}
 K_{n}/K_{n-1}=M_1^{\prime}/K_{n-1}\cdots M_{d^{n-1}}^{\prime}/K_{n-1}
 \text{ and }
 M_i^{\prime}\cap M_j^{\prime}=K_{n-1} \text{ if } i\neq j. 
 \end{equation}
By Proposition \ref{pal1},
\begin{displaymath}                                                                 \mathrm{Gal}(M_i^{\prime}/K_{n-1})\cong S_d, \text{ for } 
1\leq i\leq d^{n-1}
\end{displaymath}
Equation (\ref{bvis}) can be best visualized in the diamond diagram below:                                      
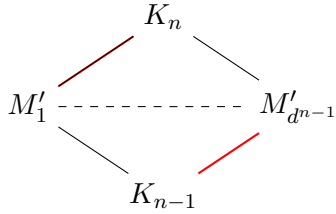
\begin{center}
\begin{tikzpicture}[scale=.6] 
  \node (30) at (-3,4) {$K_{n}$};
\node (20) at (-6,2) {${M_1^{\prime}}$};
\node (21) at (0,2)  {${M_{d^{n-1}}^{\prime}}$};
\node (10) at (-3,0) {$K_{n-1}$};
\draw [dashed](20) -- (21);
\draw [thick, red](30) -- (20);
\draw [thick, red] (21)-- (10);
\draw (10)--(20) --(30)--(21);
\end{tikzpicture}
\end{center}
By [Lang, \cite{lan}, VI,\S 1.14, 1.15],
\begin{displaymath}
 \mathrm{Gal}(K_{n}/K_{n-1})\cong \mathrm{Gal}(M_1^{\prime}/K_{n-1})\times\dots\times(M_{d^{n-1}}^{\prime}/K_{n-1})
\end{displaymath}
\begin{displaymath}
 \sigma\mapsto (\sigma|_{M_1^{\prime}},\dots,\sigma|_{M_{d^{n-1}}^{\prime}}) \text{ is an isomorphism.}
\end{displaymath}
This gives,
\begin{equation}\label{subprf}
  \mathrm{Gal}(K_{n}/K_{n-1})\cong S_{d}^{d^{n-1}} \text{ for } 1\leq i\leq d^{n-1}.
 \end{equation}
  Thus,  $ \mathrm{Gal}(K_{n}/K_{n-1})$  is maximal for all $n\geq 1$.
 \end{proof}
We know that,
\begin{displaymath}[K_{n}:k( t)]=[K_{n}:K_{n-1}][K_{n-1}:K_{n-2}]\cdots[K_1:k(t)]
 \end{displaymath}
From (\ref{subprf}),
\begin{displaymath}
 [K_{n}:k(t)]=(d!)^{d^{n-1}}  \cdot(d!)^{d^{n-2}}\cdots (d!)^{d} \cdot  d!   =|[S_d]^{n}| 
\end{displaymath}
  Thus
 \begin{equation}\label{linear-gal}
 \mathrm{Gal}(K_{n}/k(t))\cong [S_d]^{n} \text{ for all } n\geq1.
\end{equation} 
Therefore  $\mathrm{Gal}(K_{n}/k(t))$ is maximal for $n\geq 1$.
This  completes the proof of  Theorem \ref{S_n:ff}.
\section{Applications}\label{mo_nmo}
In this section we consider both Morse and non-Morse polynomials in $\F_q[t][x]$
and see that most  polynomials in $\F_q[t][x]$ 
have the property  that the Galois group of the field extension formed
by adjoining the roots of the $n$-th iterate is as large as possible.
\begin{definition} [Geyer \cite{mj}]
Let $K$ be a field of characteristic $p\geq 0.$ Suppose $f\in K[x]$, then $f$ is   a Morse function, if: 
\begin{enumerate}
 \item  the critical points of $f$, i.e., the zeros of $f^{\prime}$ are non degenerate. i.e., $f^{\prime}$ has simple roots. i.e., $f^{\prime}(\tau)=0$ implies 
 $f^{\prime\prime}(\tau)\neq 0.$
 \item  the critical values of $f$ are distinct. i.e.,
 \begin{displaymath}
  f^{\prime}(\eta)=f^{\prime}(\tau)=0 \text{ and } f(\eta)=f(\tau)\implies
  \eta=\tau
 \end{displaymath}
 \item If $p>0$, then $f(x)$ has no pole of
 order divisible by $p.$
\end{enumerate}
\end{definition}
 A basic fact 
is that   any polynomial with  generic choice of coefficients  is Morse.
 \begin{theorem}\label{jpth}[Serre \cite{JP},Theorem 4.4.5]
 If $K$ is any field of characteristic $p\geq 0$,
not dividing $n$, and $f(X)\in K[X]$ is a Morse function of degree $n$, then for any indeterminate $T$ the group $\mathrm{Gal}(f(X)-T,K(T))=S_n.$
\end{theorem}
Let $K$ be a Hilbertian field and $g(x)\in K[x]$ be any arbitrary polynomial of degree $d>0$. Therefore for an indeterminate $S$, the function $g(x)+Sx$ is Morse, (as well as separable). By Theorem \ref{jpth}, for an indeterminate $T$ over $K$, the group
\begin{equation}\label{2indet}
 \mathrm{Gal}(g(x)+Sx+T, K(S,T))\cong S_d.
\end{equation}
Recall that, every irreducible polynomial over a finite field is separable. Let us consider the following Proposition to illustrate (\ref{2indet}) over $k(t)$.
\begin{proposition}\label{su_2in}
 Let $a(t),b(t)$ be irreducible polynomials of bounded degree in $\F_q[t]$ such that
 $ f(t,x)= x^{d}+a(t)x^{d-1}+b(t) \in \F_q[t][x] $ is  irreducible and  separable over $k(t)$ of degree $d>2$ in $x$ and $p\nmid d(d-1)$.
Then the group $\mathrm{G}=\mathrm{Gal}(f(t,x), k(t))\cong S_d.$
\end{proposition}
\begin{proof}
Let $L$ be the splitting field of $x^{d}+a(t)x+b(t)$ over $k(t)$. By our hypothesis on $f(t,x)$,
the  group $
\mathrm{Gal}(x^d+a(t)x+b(t),k(t))$ is a transitive subgroup of  $S_d$.
  The function $x^d+a(t)x$ is Morse, hence the zeros of derivative $f^{\prime}(t,x)=dx^{d-1}+a(t)$ have multiplicity one.
Let
 $\gamma_1,\dots,\gamma_{d-1}$ be the  $d-1$ distinct roots of $f^{\prime}({t,x})=0$ in characteristic $p$.
  The corresponding ramification points 
  in characteristic $p$
  are $f(t,\gamma_1),\dots,f(t,\gamma_{d-1})$.
  At each $f(t,\gamma_i)$, the ramification index 
 $e(Q_{\gamma_i}|P_{f(t,\gamma_i)})$ of any prime $Q_{\gamma_i}$ of $L$ dividing the prime $P_{f(t,\gamma_i)}$ of $k[t]$ is cyclic of order $2$. Therefore the   inertia group $I_{\gamma_i}$ is tame, generated by a transposition, hence not contained in the alternating group $A_d$. 
The ramification index of any prime $Q_{\infty}$ of $L$ dividing $P_{\infty}$ is  $d$, i.e., $e(Q_{\infty}|P_{\infty})=d$. The corresponding inertia group $I_1$ is tame,  generated by an $d$ cycle.
The group $\mathrm{G}=\mathrm{Gal}(x^{d}+a(t)x+b(t), k(t))$ is a transitive subgroup of $S_d$ generated by transpositions.
Therefore by, [Lemma 4.4.4, \cite{JP}]   $\mathrm{G}$ is  primitive. So, $\mathrm{G}$ is a primitive permutation group containing transpositions. Hence $\mathrm{G}\cong S_d$.
\end{proof}
\begin{proposition}\label{S_n:fd}
 Under  the hypothesis of Proposition \ref{su_2in}, and with the above notations, the Galois group of the $n$-th iterate of $f(t,x)=x^{d}+a(t)x+b(t)$ over $k(t)$ is $[S_d]^{n}$ for all $n\geq 1,$ i.e.,
 \begin{displaymath}
  \mathrm{Gal}(f^{\circ n}(t,x),k(t))\cong 
  [S_d]^n \text{ for all } n\geq 1.
 \end{displaymath}
\end{proposition}
\begin{proof}
We may note that the $d-1$ critical points of $f(t,x)$ are distinct and have multiplicity one. Along with this fact, rest of the the proof is identical to 
Proposition \ref{sub_2indet}, below.  Therefore,
$K_{n}/K_{n-1}$ is maximal for all $n\geq 1.$
\end{proof}
\subsection{Non-Morse polynomials}\label{non_mors}
Let us consider an example of a polynomial in $\F_q[t][x]$ which has only one critical point of multiplicity one. 
 Suppose $k$ is an algebraically closed field of characteristic $p> 0$. Cohen \cite{coh} proved that, given two indeterminates $S$ and $T$ the
 Galois group of the generic trinomial $x^n+Sx^{a}+T$ over $k(S,T)$ is $S_n$ unless $p$ divides $n(n-1)$, $\mathrm{gcd}(n,a)=1$, ($n>a>0$) and  $a=1$ or $n-1$. 
 Hence,
\begin{equation}\label{hol}
  \mathrm{Gal}(x^{n}+Sx^{n-1}+T,k(S,T))\cong S_n.
\end{equation}
Therefore for $k=\overline{\F}_q$, the group
\begin{displaymath}
 \mathrm{Gal}(x^{d}+Sx^{d-1}+T,k(S,T))\cong S_d
\end{displaymath}
Applying a classical result due to Debes and Legrand [Lemma 4.2 \cite{pf}],  there exists infinitely many specializations in $\F_q[t]$  for $(S,T)\mapsto (a(t), b(t))\in \F_q[t]^{2}$ such that,
\begin{equation}\label{eqnox}
 \mathrm{Gal}(x^{d}+a(t)x^{d-1}+b(t),k(t))\cong S_d.
\end{equation}
We  prove (\ref{eqnox}) and establish Odoni's conjecture in this case.
\begin{proposition}\label{eqnox_1}
 Let $a(t),b(t)$ be irreducible polynomials of bounded degree in
 $\F_q[t]$ such that
$f(t,x)= x^{d}+a(t)x^{d-1}+b(t)\in  \F_q[t][x]$ is   irreducible and  separable over $k(t)$ of degree $d>2$ in $x$. Suppose  $d$ is an odd integer and $p\nmid d(d-1)$.
Then the group $\mathrm{Gal}(f(t,x), k(t))\cong S_d.$
\end{proposition}
\begin{proof}
Let $L$ be the splitting field of $f(t,x)=x^{d}+a(t)x^{d-1}+b(t)$ over $k(t)$. By our hypothesis on $f(t,x)$ the group $\mathrm{Gal}(f(t,x),k(t))$ is a transitive subgroup of $S_d$.
The ramified primes of the Galois extension $L/k(t)$ are the pole $P_{\infty}$ of $t$ and the places corresponding to the zeros of the derivative of $f(t,x)$ $\bmod{p}$.
The derivative \begin{equation}\label{den}
f^{\prime}(t,x)=dx^{d-2}(x+\frac{d-1}{d}a(t))\end{equation}
In characteristic $p$, the ramification points  are $f(t,0)$ and $f(t,\gamma)$ where, $\gamma=\frac{1-d}{d}a(t)\bmod{p}$.
 At $f(t,0)$, the ramification index $e(Q_{0}|P_{f(t,0)})=d-1$. Therefore the inertia group $I_1$ is tame, of order $(d-1)$. Since $d$ is an odd ineger, $I_1$ is not contained in $A_d$.
 At $f(t,\gamma)$, the ramification index $e(Q_{\gamma}|P_{f(t,\gamma)})=2$, hence,   the  inertia group $I_2$ is tame, generated by a transposition, hence not contained in $A_d$.
The ramification index of any prime $Q_{\infty}$ of $L$ dividing $P_{\infty}$ is  $d$, $Q_{\infty}$ is  totally tamely ramified. 
Thus the group $\mathrm{G}=\mathrm{Gal}(x^{d}+a(t)x^{d-1}+b(t), k(t))$ is a transitive subgroup of $S_d$ which contains a transposition  and  a $(d-1)$ cycle.
By [Lemma 4.4.3, \cite{JP}],  $\mathrm{G}$ is doubly transitive, hence primitive. Thus, $\mathrm{G}$ is a primitive permutation group containing transposition, hence $\mathrm{G}\cong S_d$.
\end{proof}
\begin{proposition}\label{sub_2indet}
 Under the hypothesis of Proposition \ref{eqnox_1},  the Galois group of the $n$-th iterate of $f(t,x)=x^{d}+a(t)x^{d-1}+b(t)$ over $k(t)$ is $[S_d]^{n}$ for all $n\geq 1$, i.e.,
\begin{equation}\label{req}
 \mathrm{Gal}(f^{\circ n}(t,x), k(t))\cong [S_d]^{n} \text{ for all } n\geq 1.
\end{equation}
\end{proposition}  
\begin{proof}
We proceed by induction. The case $n=1$ follows immediately. Suppose that the results holds true for $n-1$, i.e,   the group $\mathrm{Gal}(f^{\circ n-1}(t,x),k(t))\cong [S_d]^{n-1}.$ 
By the inductive  hypothesis and using Capelli's Lemma,  $f^{\circ n}(t,x)$ is irreducible and separable over $k(t)$ for all $n\geq 1$. 
To complete the proof it suffices to  show that, for  the unique  critical point $\gamma$ of $f(t,x)$, any nonzero prime $\mathfrak{p}_n$ of $k[t]$ dividing $f^{\circ n}(t,\gamma)$ is a squarefree primitive prime divisor which does not ramify in $K_{n}$ for $n\geq 1$. 
The critical point $\gamma=\frac{d-1}{d}a(t)$  of $f(t,x)=x^{d}+a(t)x^{d}+b(t)$
in the algebraic closure $\Omega$ of $\F_q(t)$ has  multiplicity one and $f^{\circ n}(t,\gamma)\neq 0$ for all $n\geq 1$.
\paragraph{}
Let  $\mathfrak{p}_{n}$ be  any prime of $\F_q[t]$ that divides $f^{\circ n}(t,\gamma)$  such that $\mathrm{gcd}(b(t),\mathfrak{p}_n)=1$. 
 Thus,
$\mathfrak{p}_n$ does not divide any of $f^{\circ m}(t,0)$ for all $m\geq 1$  and, $f^{\circ m}(t,\gamma)\neq f^{\circ l}(t,0)$ for any
$m,l\leq n$.
But,
$f^{\circ n}(t,\gamma)=f^{\circ n-1}(f(t,\gamma))$ is squarefree. Therefore,   $\mathfrak{p}_n$ is a square free prime divisor of  $f^{\circ n}(t,\gamma)$. 
We know  that  $K_n$ is obtained from $K_{n-1}$ by adjoining the roots of $f(t,x)-a_i$, for each root $a_i$ of $f^{\circ n-1}(t,x)=0$, for $1\leq i\leq d^{n-1}$ and,
\begin{equation}\label{cheker_2}
  f^{\circ n}(t,x)=  f^{\circ n-1}(f(t,x))=\prod_{i=1}^{d^{n-1}}f(t,x)-a_i
\end{equation}
Therefore by (\ref{cheker_2}), 
  $f^{\circ n}(t,\gamma)$  and $f^{\circ m}(t,\gamma)$  are relatively prime for all $ 0<m<n$.
 Thus
 $\mathfrak{p}_n$ is a squarefree primitive prime divisor of $f^{\circ n}(t,\gamma)$. Now, it remains to show that $\mathfrak{p}_n$ does not ramify in $K_{n}$. For $i\neq j$,  the discriminants $\mathrm{disc}_x(f(t,x)-a_i)$ and 
$\mathrm{disc}_x(f(t,x)-a_j)$ are relatively prime in $K_{n-1}[x]$. 
Suppose the 
prime $\mathfrak{q}_i$ of $K_{n-1}$ extends  to $f(t,\gamma)-a_i$
in $K_n$. The prime $\mathfrak{q}_i$ is relatively prime with the  prime 
$\mathfrak{q}_j$ of $K_{n-1}$ which extends to the ideal $f(t,\gamma)-a_j$
in $K_n$ for $i\neq j$, hence does not ramify in the splitting field of $f(t,x)-a_j$ for $i\neq j$.
\paragraph{} Suppose $\mathfrak{q}$ be any nonzero prime ideal of $K_{n-1}$ lying above $\mathfrak{p}_n$ and the lifting of $\mathfrak{q}$ to $K_{n}$ be the ideal $f(t,\gamma)-a_i$ in $K_n$. 
It is immediate that
$\mathfrak{q}$ does not ramify in the splitting field of $f(t,x)-a_j$ for $i\neq j.$
Therefore the inertia group of every prime ideal of $K_n$ lying over $\mathfrak{q}$ is trivial. 
Since, each such $\mathfrak{q}$ lying over $\mathfrak{p}_n$ does not ramify in $K_n$, it follows that $\mathfrak{p}_n$ does not ramify in $K_n$.
Thus  the polynomial $f(t,x)=x^{d}+a(t)x^{d-1}+b(t)$ satisfies the hypothesis of Theorem 
\ref{S_n:ff}. 
Therefore, 
$K_{n}/K_{n-1}$ is maximal for all $n\geq 1$. This completes the proof.
\end{proof}

\bibliographystyle{plain}

\end{document}